\documentclass[a4paper,11pt]{amsart}

\usepackage[left=2.7cm,right=2.7cm,top=3.5cm,bottom=3cm]{geometry}
\usepackage{latexsym}
\usepackage{amsthm,amssymb,amsmath,amsfonts,amscd}
\usepackage{verbatim}
\usepackage[latin1]{inputenc}

 \usepackage{latexsym}
 \usepackage{longtable}

\newfont{\cyr}{wncyr10 scaled 1100}
\newfont{\cyrr}{wncyr9 scaled 1000}

\theoremstyle{plain}
\newtheorem{theorem}{Theorem}[section]
\newtheorem{proposition}[theorem]{Proposition}

\newtheorem{corollary}[theorem]{Corollary}

\theoremstyle{definition}
\newtheorem{conjecture}[theorem]{Conjecture}
\newtheorem{definition}[theorem]{Definition}
\newtheorem{assumption}[theorem]{Assumption}

\theoremstyle{remark}
\newtheorem{remark}[theorem]{Remark}

\usepackage[usenames]{color}
\definecolor{Indigo}{rgb}{0.2,0.1,0.7}
\definecolor{Violet}{rgb}{0.5,0.1,0.7}
\definecolor{White}{rgb}{1,1,1}
\definecolor{Green}{rgb}{0.1,0.9,0.2}

\newcommand{\smallmat}[4]{\bigl(\begin{smallmatrix}#1&#2\\#3&#4\end{smallmatrix}\bigr)}

\newcommand{\W}{\mathbb W}
\newcommand{\D}{\mathbb D}

\newcommand{\pwseries}[1]{[[#1]]}

\newtheorem{thm}{Theorem}[section]

\newtheorem{rmk}[thm]{Remark}

\newcommand{\Hom}{{\operatorname{Hom}}}

\newcommand{\ord}{{\operatorname{ord }}}

\newcommand{\Spec}{{\operatorname{Spec }}}

\newcommand{\SL}{{\operatorname{SL }}}
\newcommand{\GL}{{\operatorname{GL}}}

\newcommand{\Gal}{{\operatorname{Gal}}}

\newcommand{\Sel}{{\operatorname{Sel}}}

\newcommand{\calO}{{\mathcal{O}}}

\newcommand{\calX}{{\mathcal{X}}}

\def\C{\mathbb{C}}
\def\D{\mathbb{D}}

\def\M{\mathrm{M}}

\def\Q{\mathbb{Q}}
\def\R{\mathbb{R}}
\def\T{\mathbb{T}}

\def\Z{\mathbb{Z}}

\include{thebibliography}

\begin{document}

\title{The $p$-adic variation of the Gross-Kohnen-Zagier Theorem}
\today 
\author{Matteo Longo, Marc-Hubert Nicole}

\begin{abstract} We relate $p$-adic families of Jacobi forms to Big Heegner points constructed by B. Howard, in the spirit of the Gross-Kohnen-Zagier theorem. We view this as a $\GL(2)$ instance of a $p$-adic Kudla program.
\end{abstract}

\subjclass[2000]{}
\keywords{}
\maketitle


 \section  {Introduction}In their seminal paper \cite{GKZ}, Gross, Kohnen and Zagier showed that 
Heegner points are generating series of a Jacobi form arising from as a theta lift.
Suppose that $E/\Q$ is an elliptic curve of conductor $M$. 
We assume that the sign of the functional equation of $L(E/\Q,s)$ at $s=1$ is $-1$. Then it is well known (\cite{GZ}, \cite{Kol2}, \cite{Kol1}) that if 
the value $L'(E/\Q,s)_{s=1}$ of the first derivative of $L(E/\Q,s)$ at $s=1$ is non-zero, then the $\Q$-vector space $E(\Q)\otimes_\Z\Q$ is one-dimensional. 
Moreover, one can show that this vector space is generated by 
a  {Heegner point $P_K$} attached to a quadratic imaginary field 
$K$ in which all primes $\ell\mid N$ are split
and such that $L'(E/K_D,s)_{s=1}\neq 0$. 
\noindent
As the imaginary quadratic field $K$ varies, it is 
then a natural question to investigate the relative positions of the points 
$P_K$ on this one-dimensional line. 
The theorem of Gross-Kohnen-Zagier answers this question. For each discriminant $D$ and each residue class $r\mod{2N}$
subject to the condition that $D\equiv r^2\mod{4N}$, one defines 
a Heegner point $P_{D,r}$ (so the point $P_K$ considered above corresponds 
to one of these points, for a suitable choice of the pair $(D,r)$). 
Then \cite[Thm. C]{GKZ} shows that
the relative positions of the points $P_{D,r}$, at least under the condition that $(M,D)=1$, 
are encoded by the $(D,r)$-th Fourier coefficient of the Jacobi form 
$\phi_{f_E}$ 
coming from the theta lifting of $f_E$, 
where $f_E$ is the weight $2$ newform of level $\Gamma_0(M)$ associated with $E$ by modularity; one may express briefly this relation by saying 
that Heegner points are generating series for Jacobi forms, and 
actually \cite{GKZ} formulates an  \emph{ideal statement} in which the above relation is conjecturally extended to all coefficients of $\phi_{f_E}$, including therefore those $D$ 
such that $(D,M)\neq 1$. 
Several generalizations are available in the literature, especially by Borcherds \cite{Bor1}, using singular theta liftings, and Yuan-Zhang-Zhang \cite{YZZ} using a multiplicity one theorem for automorphic representations. In particular, these works complete \cite{GKZ} by essentially proving the ideal statement alluded to above. 

The purpose of this paper is to investigate a variant in families of the GKZ theorem, in which all objects are made to vary via $p$-adic interpolation. Note that the GKZ theorem (and the closely related famous Hirzebruch-Zagier theorem) can be seen as the historical starting point of the conjectural generalization by Kudla involving higher dimensional varieties, called the Kudla program for brief, relating for example algebraic cycles on Shimura varieties and Fourier coefficients of modular forms. We hope that our $p$-adic GKZ theorem will trigger higher dimensional generalizations, giving rise to a $p$-adic analogue of the Kudla program.

We begin by briefly explaining some of the questions motivating this work. Given a primitive branch $\mathcal{R}$ of a Hida family, Howard in \cite{Ho2} introduced certain cohomology classes in the Selmer group of Hida's Big ordinary representation attached to a Hida family, that he called Big Heegner points (we refer to \cite{Ho2} for the terminology which is not explicitly introduced here). Following Hida's strategy to construct Big Galois representations, these Big Heegner points are constructed as limits of classical Heegner points on modular curves. Their specialisations at arithmetic points of $\mathcal{R}$ of weight bigger than $2$ are known, thanks to works of Howard \cite{Ho2}, Castella \cite{Cas}, \cite{CasHeeg} and, more recently, Disegni \cite{Disegni}, to interpolate classical Heegner cycles (\cite{Nek2}, \cite{Zh}). Moreover, under suitable assumptions, the Selmer group of the Hida Big Galois representation is a $\mathcal{R}$-module of rank $1$, and Big Heegner points are non-torsion elements in this Selmer group (\cite{Ho1}, \cite{Ho2}). Thus, extending scalars to the fraction field $\mathcal{K}=\mathrm{Frac}(\mathcal{R})$ of $\mathcal{R}$, we obtain a $\mathcal{K}$-vector space of dimension $1$, and all Heegner points become proportional to a fixed generator. As in the paper \cite{GKZ}, one may thus consider the proportionality coefficients which are elements of $\mathcal{K}$. 

On the other hand, the theta correspondence can be applied to each arithmetic point in the given Hida family, and in \cite{LN-Jacobi} is shown, using methods from \cite{St}, the existence of a $p$-adic family of Jacobi forms whose specialisation at arithmetic points of the metaplectic covering of $\mathcal{R}$ gives a Jacobi form coming from a theta lift of an arithmetic point in the Hida family. In other words, \cite{LN-Jacobi} constructs a $p$-adic family of Jacobi forms $\mathbb{S}$ interpolating classical theta lifts of arithmetic points in the given Hida family; here the coefficients of this $p$-adic family of Jacobi forms $\mathbb{S}$ are elements in 
$\mathcal{K}$. 
 We refer to Section  \ref{section5.2} below or \cite{LN-Jacobi} for details.

It seems then natural to formulate a conjecture which relates the coefficients of the $p$-adic family of Jacobi forms $\mathbb{S}$ and the proportionality coefficients associated to Big Heegner points, in a way similar to the original GKZ. We can thus make the following 

\begin{conjecture}\label{conjecture} Let $\mathcal{R}$ be a primitive branch of a Hida family. There exists a Zariski open subset in $\Spec(\mathcal{R})$ where the proportionality coefficients relating Big Heegner points are equal to the Jacobi-Fourier coefficients of the theta lift $\mathbb{S}$ of the Hida family.\end{conjecture} 

A more precise form of this conjecture is presented below, under some technical assumptions, as Conjecture \ref{conjecture2}, to which the reader is referred for details. 

In this paper we do not attack Conjecture \ref{conjecture} directly (see Remark \ref{rem5.4} for a possible strategy to attach this conjecture). Instead, we present evidence in favour of this conjecture of a more local nature, in which both Big Heegner points and Jacobi-Fourier coefficients of theta lifts of classical forms of trivial character in the Hida family are related (up to explicit Euler factors) to certain $p$-adic $L$-functions constructed in \cite{BD-Hida}; as a consequence, Big Heegner points and Jacobi-Fourier coefficients of Jacobi forms are indeed related, at least locally, as predicted by Conjecture \ref{conjecture}, and are both seen as manifestation of the arithmetic of $p$-adic $L$-functions. 

We now state our results in a more precise way. 
Pick a positive even integer $2k_0$, and 
an ordinary $p$-stabilized newform $f_{2k_0}$ of level $\Gamma_0(Np)$, trivial character
and weight $2k_0$, where $N\geq1$ is an odd integer and $p\nmid N$ a prime number bigger or equal to $5$.
Let $\mathcal{X}$ be the group of continuous endomorphisms of $\Z_p^\times$; we view $\Z$ inside $\mathcal{X}$ 
via the map $k\mapsto [x\mapsto x^{k-2}]$. 
Let \[f_\infty(\kappa)=\sum_{n\geq 1}a_n(\kappa)q^n\] be the Hida family 
passing through $f_{2k_0}$, where $\kappa\mapsto a_n(\kappa)$ are 
$p$-adic analytic functions defined in a neighborhood $U$ of $2k_0$ in $\mathcal{X}$; 
if we denote by $\mathcal{A}_{2k_0}$ the subring of $\bar\Q_p[\![X]\!]$ consisting of power series converging in a neighborhood of $2k_0$, then $a_n\in\mathcal{A}_{2k_0}$. 
For each even integer $2k\in U$, let $f_{2k}$ be the weight $2k$ modular form with 
trivial character appearing in the Hida family.  
The form $f_{2k}$
is then an ordinary $p$-stabilized form of weight $2k$ and level $\Gamma_0(Np)$, 
and $f_{2k_0}$ is the form we started with above, justifying the slight abuse of notation.
If 
$f_{2k}$ is not a newform (which is always the case if $k>1$), 
then we write $f_{2k}^\sharp$ for the ordinary newform of weight $2k$ and level $\Gamma_0(N)$ whose $p$-stabilization is $f_{2k}$. 
For each of the forms $f_{2k}^\sharp$ we have an associated Jacobi form 
\[\mathcal{S}_{D_0,r_0}({f_{2k}^\sharp})=\sum_{D=r^2-4Nn}c_{f_{2k}^\sharp}(n,r)q^n\zeta^r\]
of weight $k+1$, index $N$ 
having the same Hecke eigenvalues as 
$f_{2k}^\sharp$;  
this form depends on the choice of a pair of integers $(D_0,r_0)$
such that $D_0$ is a fundamental discriminant which we assume to be prime to $p$ and $D_0\equiv r_0^2\mod 4N$.
To be clear, in the above formula the sum is over all negative discriminants $D$, and the coefficients 
$c_{f_{2k}^\sharp}(n,r)$ only depend on the residue class of $r\mod{2N}$.  
The first observation we make is that these coefficients $c_{f_{2k}^\sharp}(n,r)$ can be interpolated, up to Euler factors, by $p$-adic analytic functions. We show that a suitable linear combination 
$\mathcal{L}_{n,r}\in\mathcal{A}_{2k_0}$ 
of the square root $p$-adic $L$-functions 
attached to 
genus characters of real quadratic fields, introduced in \cite{BD}, \cite{BD-Hida} and studied extensively in \cite{Sha}, \cite{GSS}, \cite{LV-IMRN}, interpolates Fourier-Jacobi 
coefficients $c_{f_{2k}^\sharp}(n,r)$. Our first result is, corresponding to Theorem \ref{LambdaJacobi} below, is:

\begin{theorem}\label{Result I} For all positive 
even integers $2k$ in a sufficiently small neighbourhood of $2k_0$, 
and for all $D=r^2-4nN$ such that $p\nmid D$, we have
$\mathcal{L}_{n,r}(2k)\overset{\bullet}=c_{f_{2k}^\sharp}(n,r)
$
where $\overset\bullet=$ means equality up to an explicit Euler factor
and a suitable $p$-adic period, both independent of $(n,r)$ and 
non-vanishing. \end{theorem}

On the other hand, as recalled above, B. Howard constructed in \cite{Ho2} an analogue of the Heegner point $P_K$ in the setting of Hida families. 
To explain this, recall that attached to $f_\infty$ we have a Big Galois representation $\mathbb{T}^\dagger$ interpolating self-dual twists of the Deligne representation attached to $f_{2k}$ for $2k$ a positive even integer.
If we denote (as before) by $\mathcal{R}$ the branch of the Hida-Hecke algebra corresponding to $f_\infty$ (which is a complete noetherian integral domain, finite over $\Lambda=\Z_p[\![1+p\Z_p]\!]$), 
then $\mathbb{T}^\dagger$ is a free $\mathcal{R}$-module of rank $2$ 
equipped with a continuous action of 
$G_\Q=\Gal(\bar{\Q}/\Q)$ and 
specialization maps $\mathbb{T}^\dagger\mapsto \mathbb{T}^\dagger_{\kappa}$ for each $\kappa\in\mathcal{X}$, 
such that when $\kappa=2k$ is a positive even integer,
$\T_{2k}^\dagger$ is isomorphic 
to the self-dual twist of the Deligne representation attached to $f_{2k}$. We assume throughout this paper that the residual representation 
$\T^\dagger/\mathfrak{m}_\mathcal{R}\T^\dagger$ (where $\mathfrak{m}_\mathcal{R}$ is the maximal ideal of $\mathcal{R}$) is absolutely irreducible and $p$-distinguished. 
Let $\Sel(\Q,\mathbb{T}^\dagger)\subseteq H^1(G_\Q,\mathbb{T}^\dagger)$ be the Greenberg 
Selmer group attached to the Galois representation $\mathbb{T}^\dagger$. 
We require that the generic sign of the 
functional equation of the $L$-functions of $f_{2k}$ is $-1$, 
and that the central critical value of their 
first derivative is generically non-vanishing, 
{cf.} Assumptions \ref{ass1} and \ref{ass2} below. 
Under these conditions, which might be viewed as 
analogues to those in \cite{GKZ}, the $\mathcal{R}$-module 
$\Sel(\Q,\mathbb{T}^\dagger)$ is finitely generated 
of rank $1$. 
Howard constructs certain cohomology classes \[\mathfrak{Z}_{D,r}^\mathrm{How}
\in\Sel(\Q,\mathbb{T}^\dagger),\] taking inverse limits of norm-compatible sequences of 
Heegner points in towers of modular curves. 
Therefore it is a natural question to relate the positions of the points 
$\mathfrak{Z}_{D,r}^\mathrm{How}$ in the $1$-dimensional $\mathcal{K}$-vector space
$\Sel_\mathcal{K}(\Q,\mathbb{T}^\dagger)=\Sel(\Q,\mathbb{T}^\dagger)\otimes_{\mathcal{R}}\mathcal{K}$, 
where $\mathcal{K}=\mathrm{Frac}(\mathcal{R})$ is the fraction field of $\mathcal{R}$. To describe our second result, we also require that the height pairing between Heegner cycles is 
positive definite, cf. Assumption \ref{ass3} below.
Define 
\[\mathcal{Z}_{n,r}(\kappa)=2u_{D}\cdot(2D)^{\frac{\kappa-2}{4}}\cdot \mathfrak{Z}_{D,r}^\mathrm{How}(\kappa)\] 
for $D=r^2-4Nn$ and 
$\kappa\in\mathcal{X}$; here $2u_{D}$ is the number of units in $\Q(\sqrt{D})$ and, 
as above, all primes dividing $D=r^2-4Nn$ are required to 
split in $K_D$.  Suppose $2k_0\equiv 2\mod{p-1}$. 
Let $\mathcal{M}_{2k_0}$ be the fraction field of $\mathcal{A}_{2k_0}$. There is a canonical map $\mathcal{K}\rightarrow\mathcal{M}_{2k_0}$, and we may define 
$\Sel_{\mathcal{M}_{2k_0}}(\Q,\T^\dagger)= 
\Sel_\mathcal{K}(\Q,\T^\dagger)\otimes_{\mathcal{K}}
{\mathcal{M}_{2k_0}}$. 
Our second result, under Assumptions \ref{ass1}, \ref{ass2}, \ref{ass3} discussed above, is the following: 

\begin{theorem}\label{Result II} There exists an element $\Phi^\text{\'et}\in 
\Sel_{\mathcal{M}_{2k_0}}(\Q,\mathbb{T}^\dagger)$ 
such that in a sufficiently small connected neighborhood of $2k_0$ in $\mathcal{X}$ and
for all $D=r^2-4nN$ such that $p$ splits in $\Q(\sqrt{D})$ we have
$
\mathcal{Z}_{n,r}=\mathcal{L}_{n,r}\cdot\Phi^\text{\'et}.$
\end{theorem}
This corresponds to Theorem \ref{AnalyticThm} below. 
Combining Theorems \ref{Result I} and \ref{Result II}, we obtain a partial evidence toward Conjecture \ref{conjecture}: 

\begin{corollary}\label{corointro}
For all positive 
even integers $2k$ in a sufficiently small connected neighbourhood of $2k_0$, 
and for all $D=r^2-4nN$ such that  $p$ splits in $\Q(\sqrt{D})$, we have
\[\mathcal{Z}_{n,r}\overset{\bullet}=c_{f_\kappa^\sharp}(n,r)\cdot \Phi^\text{\'et}\] 
where $\overset{\bullet}=$ means equality up to simply algebraic factors independent of $(n,r)$. 
\end{corollary}
See \eqref{GKZ1} and \eqref{GKZ4} for a more precise version of this result, including all the algebraic factors involved. 

\begin{remark} As remarked above, we may view Corollary \ref{corointro} as a fragment of Conjecture \ref{conjecture}. 
However, it should be noticed that the coefficients we are considering in this work and in \cite{LN-Jacobi} are slightly different. In this paper we consider coefficients of Jacobi forms which are lifts of modular newforms of level prime to $p$ and trivial character, whose ordinary $p$-stabilisation belong to our Hida family; the discrepancy between newforms and their $p$-stabilisations is the origin of the Euler factors (denoted $\mathcal{E}_p$ and defined in \eqref{Ep} below) relating $p$-adic $L$-functions with Heegner points and Jacobi-Fourier coefficients of Jacobi forms in a way similar to \cite{Darmon-Tornaria}; this should clarify the crucial role played in this paper by the $p$-adic $L$-functions from \cite{BD-Hida}.  On the other hand, in \cite{LN-Jacobi} we consider Jacobi forms lifting the members of the Hida family $f_{2k}$ (or more generally $f_\kappa$ for an arithmetic point $\kappa$, as in \cite{St}). The theta lifts used in the two papers are then different, since in this paper the level of the forms is prime to $p$, while in \cite{LN-Jacobi} the level is divisible by $p$. The Jacobi-Fourier coefficients coming from different theta liftings can be directly related by means of Euler factors (denoted $\mathcal{E}_2$ and defined in \eqref{E2} below). Since these Euler factors are independent of $(n,r)$, the quotients of Jacobi-Fourier coefficients coming from the two theta liftings, when defined, are the same, and then in this sense our result can be seen as an evidence of Conjecture \ref{conjecture}. The details are discussed in Section  \ref{section5.2}. 
 \end{remark}
 
%

\section   *{Acknowledgments} Part of this work has been done during visits of 
M.-H.N. at the Mathematics Department of the University of Padova, whose great hospitality he is grateful for. The paper was finalized during a visit of M.L. in Montr\'eal supported by the grant of the CRM-Simons professorship held by M.-H.N. in 2017-2018 at the Centre de recherches math\'ematiques (C.R.M., Montr\'eal). Both authors thank S. Zemel for useful email exchanges. 
 
\section  {theta lifts} \label{sec:Jacobi} 
In this section, we recall the formalism of theta liftings relating 
elliptic and Jacobi cuspforms, following \cite{EZ}, \cite{GKZ}.   

Fix an odd integer $N\geq 1$ and $k$ a positive even integer $2k$.   
Denote $S_{2k}(\Gamma_0(N))$ the complex vector space of 
cuspforms of weight $2k$ and level $\Gamma_0(N)$
and $J^\mathrm{cusp}_{k+1,M}$ the complex vector space of 
Jacobi cuspforms of weight $k+1$ and index $N$ (see \cite[Ch. I, \S 1]{EZ} for a definition).
For $f\in S_{2k}(\Gamma_0(N))$ and $\phi\in J^\mathrm{cusp}_{k+1,N}$ 
we write the corresponding $q$-expansions   
$f(z)=\sum_{n\geq 1}a_nq^n$ and $(q,\zeta)$-espansions 
\[\phi(\tau,z) = \sum_{r^2-4Nn<0} c(n,r) q^n \zeta^r,\] 
where $q = e^{2 \pi i \tau}$ and $\zeta = e^{2 \pi i z}$ 
and the second sum is over all pairs $(n,r)$ of integers such that 
$r^2-4Nn<0$. 

\begin{definition}\label{index pairpair}
A  {level $N$ index pair} is a pair $(D,r)$ of integers 
consisting of a negative discriminant $D$ of an integral quadratic form 
$Q=[a,b,c]$ such that 
$D\equiv r^2\mod 4N$. A level $N$ index pair
is said to be  {fundamental} if $D$ is a fundamental discriminant. 
\end{definition}

The Fourier expansion of 
forms $\phi\in J^\mathrm{cusp}_{k+1,N}$ is enumerated 
by level $N$ index pairs, explaining the terminology. 
Since $N$ is fixed throughout the paper, unless otherwise stated we simply call  {index pairs} or  {fundamental index pairs} 
the pairs $(D,r)$ in Definition \ref{index pairpair}. 
If $(D,r)$ is a index pair, then we usually denote $n$ the integer 
such that $D^2=r^2-4Nn$.
The spaces $S_{2k}(\Gamma_0(N))$ and $J_{k+1,N}^\mathrm{cusp}$ are equipped with 
the action of standard Hecke operators $\mathrm{T}(m)$ and $\mathrm{T}_J(m)$ respectively, for integers $m\geq 1$. 
We recall the formula from \cite[Thm. 4.5]{EZ} for the action of 
$\mathrm{T}_J(m)$ on 
Jacobi cuspforms when $(N,m)=1$.  
If $D=r^2-4Nn$ is a fundamental discriminant, 
$\phi=\sum_{n,r}c(n,r)q^n\zeta^r$ belongs to $J_{k+1,N}^\mathrm{cusp}$ 
and $\phi|\mathrm{T}_J(m)=\sum_{n,r}c^*(n,r)q^n\zeta^r$, then we have 
\begin{equation}\label{Fourier}
c^*(n,r)=\sum_{d\mid m}d^{k-1}\left(\frac{D}{d}\right)c\left(\frac{nm^2}{d^2},\frac{rm}{d}\right).
\end{equation}

For any ring $R$, let $\mathcal P_{k-2}(R)$ denotes 
the $R$-module of homogeneous polynomials in 2 variables 
of degree $k-2$ with coefficients in $R$, 
equipped with the right action of the semigroup $\M_2(R)$ defined by the formula 
\begin{equation}\label{action}
(F|\gamma)(X,Y)=F\left(aX+bY,cX+dY\right)\end{equation}
for $\gamma=\smallmat abcd$. Let $\mathcal{V}_{k-2}(R)$ denote the $R$-linear dual of 
$\mathcal{P}_{k-2}(R)$, equipped with the left $\M_2(R)$-action induced from that on 
$\mathcal{P}_{k-2}(R)$. 

Let $f$ be a cuspform of integral even weight $2k$ and level $\Gamma_0(N)$.
To $f$ we may associate the modular symbol 
$\tilde{I}_f\in\mathrm{Symb}_{\Gamma_0(N)}(\mathcal{V}_{2k-2}(\C))$ by the integration formula 
\[\tilde{I}_f\{r\rightarrow s\}(P)=2\pi i\int_r^sf(z)P(z,1)dz.\] 
Here, for any congruence subgroup $\Gamma\subseteq\SL_2(\Z)$ 
and any $\Z[\GL_2(\Q)\cap\M_2(\Z)]$-module $M$, we denote $\mathrm{Symb}_\Gamma(M)$ 
the group of $\Gamma$-invariant modular symbols with values in $M$ (cf. \cite[(4.1)]{GS}). 
The matrix $\smallmat 100{-1}$ normalizes 
$\Gamma_0(N)$ and hence induces an involution on the space of modular symbols 
$\mathrm{Symb}_{\Gamma_0(N)}(\mathcal{V}_{2k-2}(\C))$; for each $\varepsilon\in\{\pm1\}$, 
we denote $\tilde{I}_f^\varepsilon$ the 
$\varepsilon$-eigencomponents of $\tilde{I}_f$ with respect to this involution. 
It is known that there are complex periods $\Omega_f^\varepsilon$ such that 
\[I_f^\varepsilon=\frac{\tilde{I}_f^\varepsilon}{\Omega_f^\varepsilon}\] belong 
to $\mathrm{Symb}_{\Gamma_0(N)}(\mathcal{V}_{2k-2}(F_f))$, where 
$F_f$ is the extension of $\Q$ generated by the Fourier coefficients of $f$. These periods 
can be chosen so that the Petersson norm $\langle f,f\rangle$ equals the product 
$\Omega_f^+\cdot\Omega_f^-$; note that the 
$\Omega_f^\varepsilon$ are well-defined 
only up to multiplication by non-zero factors in $F^\times_f$. 

For each integer $\Delta$, let 
$\mathcal{Q}_{\Delta}$ be the set of integral quadratic 
forms  \[Q=[a,b,c]=ax^2+bxy+cy^2\]
of discriminant $\Delta$ and, for any integer $\rho$, let
$\mathcal{Q}_{N,\Delta,\rho}$ denote the subset of $\mathcal{Q}_\Delta$ consisting 
of integral 
binary quadratic forms $Q=[a,b,c]$ of discriminant $\Delta$ such that 
$b\equiv\rho\pmod{2N}$ and $a\equiv 0\mod N$. Let 
$\mathcal{Q}^0_{N,\Delta,\rho}$ be the subset of 
$\mathcal{Q}_{N,\Delta,\rho}$ consisting of forms which are $\Gamma_0(N)$-primitive, 
i.e., those $Q\in \mathcal{Q}_{N,\Delta,\rho}$ 
which can be written as $Q=[Na,b,c]$ with $(a,b,c)=1$. 
These sets are equipped with the right action 
of $\SL_2(\Z)$ described in \eqref{action}. 

Fix a fundamental index pair $(D_0,r_0)$. For any index pair 
$(D,r)$, 
let $\mathcal F_{D_0,r_0}^{(D,r)}(N)$ be the set of 
integral binary quadratic forms 
$Q=[a,b,c]$ 
modulo the right action of $\Gamma_0(N)$
described in \eqref{action},
such that: 
\begin{itemize} 
\item $\delta_Q=b^2-4ac=D_0D$; 
\item $b\equiv -r_0r\mod 2N$;
\item $a\equiv 0\mod N$. 
\end{itemize} 



Let $Q\mapsto\chi_{D_0}(Q)$ be the generalized genus character attached to $D_0$
defined in \cite[Prop. 1]{GKZ}. We recall the definition 
for $Q\in \mathcal{Q}_{N,\Delta,\rho}$. If $Q=\ell\cdot Q'$ for some form $Q'\in\mathcal{Q}^0_{N,\Delta,\rho}$, 
then define $\chi_{D_0}(Q)=\left(\frac{D_0}{\ell}\right)\cdot \chi_{D_0}(Q')$, so it is enough 
to define it on $\Gamma_0(N)$-primitive forms. Fix $Q\in\mathcal{Q}^0_{N,\Delta,\rho}$.  
If $(a/N,b,c,D_0)=1$, then pick any factorization $N=m_1\cdot m_2$ with $m_1>0$, $m_2>0$ 
and any integer $n$ coprime with $D_0$ 
represented by the quadratic form $[a/m_1,b,cm_2]$; then put  
$\chi_{D_0}(Q)=\left(\frac{D_0}{n}\right)$. If $(a/N,b,c,D_0)\neq1$, then put 
$\chi_{D_0}(Q)=0$. 
In the previous notation, if $Q$ is a representative 
of a class in 
$\mathcal F_{D_0,r_0}^{(D,r)}(N)$, then $Q$ belongs to 
$\mathcal{Q}_{M,D_0D,-r_0r}$. In particular, we may consider the 
genus character $Q\mapsto\chi_{D_0}(Q)$ for all classes $Q$ 
in $\mathcal F_{D_0,r_0}^{(D,r)}(N)$.

Fix $Q=[a,b,c]\in\mathcal{Q}_{M,\Delta,\rho}$. Following \cite{GKZ}, \cite{St}, we define certain geodesic lines in $\mathcal{H}$, with respect to the 
Poincar\'e metric on $\mathcal{H}$, as follows. 
If $\delta_Q=m^2$ ($m>0$) is a perfect square, let $C_Q$ be the geodesic line 
from $(-b-m)/2a$ to $(-b+m)/2a$ when $a\neq 0$, while if $a=0$, let $C_Q$ be 
the geodesic line from $-c/b$ to $i\infty$ if $b>0$ and from $i\infty$ to $-c/b$ if $b<0$. If $\delta_Q$ is not a perfect square, we denote by $\gamma_Q$ the matrix in $\SL_2(\Z)$ corresponding to a unit of the quadratic form $Q$ and denote $C_Q$ 
the geodesic between $z_0$ and $\gamma_Q(z_0)$, where $z_0$ is any point in $\mathbb{P}^1(\Q)$ (we can take $z_0=i\infty$ for example). 
We let $r_Q$ and $s_Q$ denote the source and the target of the geodesic line $C_Q$. Note that in any case $s_Q$ and $r_Q$ belong to $\mathbb{P}^1(\Q)$, 
and therefore the modular symbol $I_f^-\{r_Q\rightarrow s_Q\}$ is well-defined.   

For $D=r^2 - 4Nn$, define  
\begin{equation}\label{Shintani}
{c}_f\left(n,r\right)= 
\sum_{Q\in\mathcal F_{D_0,r_0}^{(D,r)}(N)}\chi_{D_0}(Q)\cdot
{I}_f^-\{r_Q\rightarrow s_Q\}(Q^{k-1})\end{equation}
and set 
\[{\mathcal{S}}_{D_0,r_0}(f)=\sum_{r^2 - 4Nn<0}{c}_f(n,r) q^n \zeta^r.\] 
Let $S_{2k}^-(\Gamma_0(N))$ be the subspace of $S_{2k}(\Gamma_0(N))$
consisting of forms whose $L$-function admits a functional equation with 
sign $-1$. 
The association $f\mapsto {\mathcal{S}}_{D_0,r_0}(f)$ gives a $\C$-linear map, 
called  {theta lifting}, 
\[{\mathcal{S}}_{D_0,r_0}:S_{2k}^-(\Gamma_0(N))
\longrightarrow J_{k+1,N}^\mathrm{cusp}\]
which is equivariant for the action of Hecke operators on both spaces, 
and such that $c(n,r)$ belong to $F_f$ for all $n,r$; 
see \cite[Ch. II]{EZ} for details. The restriction of ${\mathcal{S}}_{D_0,r_0}$ 
to newforms of $S_{2k}^-(\Gamma_0(N))$ is an
isomorphism onto the image, and for different choices of $(D_0,r_0)$ we get 
multiples of the same Jacobi form. 

\begin{remark}
The definition of the geodesic line $C_Q$ is slightly different in \cite{GKZ}, since 
it is defined to be any geodesic in the upper half plane connecting a point $z_0\in\mathcal{H}$ to the point $\gamma_Q(z_0)$. However, the value of the integral is independent of the choice of $z_0\in\mathcal{H}$, and therefore, passing to the limit, one can take $z_0$ to be any cusp as well. See also \cite[\S 2.1]{St} 
for a similar definition, keeping in mind that the convention in \emph{loc. cit.} and 
in this paper (which follow more closely \cite{GKZ}) are slightly different. 
\end{remark}

\section  {$p$-adic analytic theta lifts} \label{theta-liftings}

In this section we construct a $p$-adic analytic family of Jacobi forms 
using the $p$-adic variation of integrals. 
This family interpolates Fourier-Jacobi coefficients of 
newforms of level $\Gamma_0(N)$ whose $p$-stabilizations 
belong to a fixed Hida family (in contrast with \cite{LN-Jacobi}, 
where $\Lambda$-adic families directly interpolate arithmetic specializations in Hida families).  
The results of this section are, as in \cite{LN-Jacobi}, consequences of Stevens and Hida's works on 
half-integral weight modular forms \cite{St}, \cite{HarHalf} via the formalism of theta liftings 
and the work of Greenberg-Stevens \cite{GS} and Bertolini-Darmon \cite{BD-Hida} 
on measure-valued modular symbols.  
Fix throughout a prime number $p\nmid N$, $p\geq 5$. 

We first set up the notation for Hida families. 
Let $\Gamma=1+p\Z_p$ and $\Lambda=\mathcal{O}[\![{\Gamma}]\!]$ its Iwasawa 
algebra, where $\mathcal{O}$ is the valuation ring of a finite field extension of $\Q_p$. 
We fix an isomorphism $\Lambda\simeq\mathcal{O}[\![X]\!]$.
Let 
\[\mathcal{X}=\Hom_\mathrm{cont}(\Z_p^\times,\Z_p^\times)\] 
be the group of continuous group $\Z_p^\times$-valued homomorphisms of $\Z_p^\times$. 
Embed $\Z$ in $\mathcal{X}$ by $k\mapsto [x\mapsto x^{k-2}]$; if we 
equip $\mathcal{X}$ with the rigid analytic topology, then 
$\Z\subseteq\mathcal{X}$ is rigid Zariski dense.  
We see elements $a\in\Lambda$ as functions on $\mathcal{X}$ by $a(\kappa)=\varphi_\kappa(a)$ 
for all $\kappa\in\mathcal{X}$, where $\varphi_\kappa:\Lambda\rightarrow\mathcal{O}$ denotes the  
$\mathcal{O}$-linear extension of $\kappa$.  

Recall the fixed $p$-stabilized form $f_0$ of trivial nebentype and weight $k_0$ and $\mathcal{R}$ be the branch of the Hida family passing through $f_0$. 
Then $\mathcal{R}$ is an integral noetherian domain, finite over $\Lambda$; we denote $\mathcal{X}(\mathcal{R})$ the set of 
continuous homomorphisms of $\mathcal{O}$-algebras $\mathcal{R}\rightarrow\bar{\Q}_p$, and $\mathcal{X}^\mathrm{arith}(\mathcal{R})$ 
the subset of arithmetic points (see \cite[Def. 2.1.1]{Ho2} for its definition).
For each positive even integer $2k\in\mathcal{X}$, using that 
$\mathcal{R}/\Lambda$ is unramified at the point $\varphi_{2k}$, 
we obtain a unique arithmetic point $\tilde\varphi_{2k}$ lying over $\varphi_{2k}$. 
Hida theory shows that there exists a formal power series $F_\infty=\sum_{n\geq 1}A_nq^n\in
\mathcal{R}\pwseries{q}$ such that for each $\varphi\in\mathcal{X}^\mathrm{arith}(\mathcal{R})$ the power series $F_\varphi=\sum_{n\geq 1}\varphi(A_n)q^n$ is the $q$-expansion of a $p$-stabilized newform, 
and $F_{\tilde\varphi_{2k_0}}=f_{0}$. 

Fix an even positive integer $2k_0$ and 
let $\mathcal{A}_{2k_0}$ be the ring of power series 
in $\bar\Q_p[\![X]\!]$ which converge in a neighborhood of $2k_0$. 
If $\mathcal{R}_{2k_0}$ denotes the localisation of $\mathcal{R}$ at 
$\tilde\varphi_{2k_0}$, then, using that $\mathcal{A}_{2k_0}$ is Henselian, 
we obtain a canonical morphism $\psi_{2k_0}:\mathcal{R}_{2k_0}\rightarrow\mathcal{A}_{2k_0}$ (\cite[(2.7)]{GS}), and therefore, using the localisation map, we have a canonical map $\mathcal{R}\rightarrow\mathcal{A}_{2k_0}$.
If $\mathcal{K}$ and $\mathcal{M}_{2k_0}$ are the fraction fields
of $\mathcal{R}$ and $\mathcal{A}_{2k_0}$, respectively, 
then we obtain a map, still denoted $\psi_{2k_0}:
\mathcal{K}\rightarrow\mathcal{M}_{2k_0}$. The domain 
of convergence about $2k_0$ is the intersections of the discs of 
convergence of $\psi_{2k_0}(a)$ for $a\in\mathcal{R}$. 
We denote $U_{2k_0}$ the domain of convergence, which we may assume to be connected. 
We let 
\[f_\infty=\sum_{n\geq 1}a_nq^n\] be the image of $F_\infty$ 
via $\psi_{2k_0}$; so $a_n=\psi_{2k_0}(A_n)$ are rigid analytic functions 
converging in the domain of convergence about $2k_0$, 
such that for each even positive integer $2k\in U_{2k_0}$ 
the power series $\sum_{n\geq 1}a_n(2k)q^n$ is the 
Fourier expansion of an ordinary $p$-stabilized modular form $f_{2k}\in S_{k}(\Gamma_0(Np))$. 

We now describe universal measure-valued modular symbols, following \cite{GS} and \cite{BD-Hida}. 
Let $\mathbb{D}_*$ denote the $\mathcal{O}$-module of $\mathcal{O}$-valued measures on $\Z^2_p$ which are supported on 
the set of primitive vectors $(\Z_p^2)'$ of $\Z_p^2$. 
The $\mathcal{O}$-module $\D_*$ is equipped with the action induced by the action of $\GL_2(\Z_p)$ on $\Z_p^2$ by $(x,y)\mapsto (ax+by,cx+dy)$ for $\gamma=\smallmat abcd\in\GL_2(\Z_p)$, and a structure of $\mathcal{O}\pwseries{\Z_p^\times}$-module induced by the scalar action of $\Z_p^\times$ on $\Z_p^2$ ({cf.} \cite[\S 5]{St}, \cite[\S 2.2]{BD}); in particular, $\D_*$ is also equipped 
with a structure of $\Lambda$-module. 

Let $\Gamma_0(p\Z_p)$ denote the subgroup of $\GL_2(\Z_p)$ 
consisting of matrices which are upper triangular modulo $p$. 
The group $\mathrm{Symb}_{\Gamma_0(N)}(\mathbb{D}_*)$ is equipped with an action of 
Hecke operators, including the Hecke operator at $p$, denoted $\mathrm{U}(p)$, 
and we denote 
\[\W=\mathrm{Symb}_{\Gamma_0(N)}^\ord(\mathbb{D}_*)\] the 
ordinary subspace of $\mathrm{Symb}_{\Gamma_0(N)}(\mathbb{D}_*)$ for the action of $\mathrm{U}(p)$; see \cite[(2.2), (2.3)]{GS} for details and definitions. 

For any integer $k\in \mathcal{X}$, we have a $\Gamma_0(p\Z_p)$-equivariant homomorphism 
$\rho_\kappa:\mathbb{D}_*\rightarrow\mathcal{V}_{k-2}(\C_p)$ defined by the formula 
\[\rho_k(\mu)(P)=\int_{\Z_p\times\Z_p^\times}P(x,y)d\mu(x,y)\]
which gives rise to an homomorphism, denoted by the same symbol,  
\[\rho_k:\W\longrightarrow 
\mathrm{Symb}_{\Gamma_0(N)}\left(\mathcal{V}_{k-2}(\C_p)\right).\] We may then 
define $\W_{{\mathcal{A}_{2k_0}}}=\W\otimes_\Lambda\mathcal{A}_{2k_0}$ and 
consider the extension of $\rho_k$, still denoted by the same symbol,
\[\rho_k:\W_{\mathcal{A}_{2k_0}}\longrightarrow 
\mathrm{Symb}_{\Gamma_0(N)}\left(\mathcal{V}_{k-2}(\C_p)\right).\] 
By \cite[Thm. 5.13]{GS}, 
there exists a connected neighborhood $U_{2k_0}$ of $2k_0$ in $\mathcal{X}$ 
and an element $\Phi_{2k_0}\in\W_{\mathcal{A}_{2k_0}}$ 
such that for all positive even integers $2k\in U_{2k_0}$ with $k\geq 1$ we have 
\begin{equation}\label{CT}
\rho_k(\Phi_{2k_0})=\lambda(k)\cdot I_{f_{2k}}\end{equation}
where $\lambda(k)\in F_{f_{2k}}$, the 
field extension of $\Q_p$ generated by the Fourier coefficients of $f_{2k}$, 
and $\lambda(k_0)=1$. 

Fix a fundamental index pair $(D_0,r_0)$ so that 
$p\nmid D_0$, and an index pair $(D,r)$ with $p\nmid D$. 
Fix a system of representatives $\mathcal{R}_{D_0,r_0}^{(D,r)}(N)$ 
of $\mathcal{F}_{D_0,r_0}^{(D,r)}(N)$ 
so that each form $[a,b,c]$ in $\mathcal{R}_{D_0,r_0}^{(D,r)}(N)$ 
satisfies $p\nmid a$ and $p\nmid c$. Such a 
system can easily be obtained up to multiplying 
quadratic forms by matrices of the form $\smallmat 1i01$ and 
$\smallmat 10{Ni}1$, for suitable integers $i$. 
Fix a matrix $Q=[a,b,c]$ in $\mathcal{R}_{D_0,r_0}^{(D,r)}(N)$. 
Let $\Delta=D_0D$. 
If $p$ is inert in $K_\Delta$, we define 
$X_Q=(\Z_p^2)'$. 
If $p$ is split in $K_\Delta$, then 
$Q(X,Y)$ splits into the product $a\cdot q_1(X,Y)\cdot q_2(X,Y)$
where $q_1(X,Y)=X-\alpha Y$ and $q_2(X,Y)=X-\beta Y$ 
are two linear forms in $\Z_p[X,Y]$. Define the elements
$v_1=(\alpha,1)$ and $v_2=(\beta,1)$  
and put 
$X_Q=\Z_p^\times\cdot v_1+\Z_p^\times v_2$. 
Since $p\nmid ac$, then $\Z_p v_1+\Z_pv_2=\Z_p^2$. 
Both in the split and inert cases, it is easy to show that
 $Q(x,y)\in\Z_p^\times$ for all $(x,y)\in X_Q$. 

Define for $\kappa\in U_{2k_0}$ and $Q\in \mathcal{R}_{D_0,r_0}^{(D,r)}(N)$, 
\[\mathcal{L}_Q(\kappa)=\int_{X_Q}\omega\left(Q(x,y)\right)^{k_0-1}\langle Q(x,y)\rangle^\frac{k-2}{2}d\Phi_{2k_0}\{r_Q\rightarrow s_Q\}(x,y).\]

The next result exploits the interpolation formulas of 
$\mathcal{L}_Q(\kappa)$. 
Put 
\begin{equation}\label{Ep}\mathcal{E}_p=
\begin{cases} \left(1-\frac{p^{k-1}}{a_p(2k)}\right)^2 & \text{if $p$ is split in $K_\Delta$},\\
1-\frac{p^{2k-2}}{a_p^2(2k)} & \text{if $p$ is inert in $K_\Delta$}. 
\end{cases}\end{equation}
Then, for all even positive integers $2k\in U_{2k_0}$ with $k>1$ 
and 
all quadratic forms $Q$ in $\mathcal{R}_{D_0,r_0}^{(D,r)}(N)$ we have 
\begin{equation}\label{interpol}
\mathcal{L}_Q(2k)=\lambda(k)\cdot\mathcal{E}_p\cdot I_{f_{2k}^\sharp}\{r_Q\rightarrow s_Q\}.\end{equation}
In the inert case, \eqref{interpol} 
is \cite[Prop. 2.4]{BD}. In the split case, 
a proof of this result 
can be found in \cite[Prop. 3.3.1]{Sha}; see also 
\cite[Sec. 5.1]{GSS}, \cite[Prop. 4.24]{LV-IMRN}, \cite[Prop. 3.3]{LM}. 

Suppose that both $D_0$ and $D$ 
are fundamental discriminants. 
With the usual convention that $D=r^2-4nN$, put
\begin{equation}\label{p-adic-L-function}
\mathcal{L}_{n,r}(\kappa)=\sum_{Q\in \mathcal{R}_{D_0,r_0}^{(D,r)}(N)}\chi_{D_0}(Q)\cdot\mathcal{L}_Q(\kappa).\end{equation}
The function $\mathcal{L}_{n,r}(\kappa)$ is called 
the  {square-root $p$-adic $L$-function} attached 
to the genus character $\chi_{D_0}$ of the real quadratic 
field $K_\Delta$,
since its square for $\kappa=2k$ an even positive integer in $U_{2k_0}$, interpolates 
special values $L(f_k^\sharp/K_\Delta,\chi_{D_0},k)$. 
More precisely, setting $L_{n,r}=\mathcal{L}_{n,r}^2$, we have 
for all even positive integers $2k\in U_{2k_0}$, 
\[L_{n,r}(2k)=\lambda(k)^2\cdot\mathcal{E}_p^2\cdot\Delta^{k-1}\cdot 
L^\mathrm{alg}(f_{2k}^\sharp/K_\Delta,\chi_{D_0},k)\]
where the algebraic part of the special value of the $L$-series 
of $f_{2k}^\sharp$ twisted by $\chi_{D_0}$ is defined as 
\[L^\mathrm{alg}(f_{2k}^\sharp/K_\Delta,\chi_{D_0},k)=\frac{(k-1)!^2\sqrt{\Delta}}{(2\pi i)^{2k-2}\cdot(\Omega^-_{f_{2k}^\sharp})^2}\cdot L(f_{2k}^\sharp/K,\chi_{D_0},k).\]
In the inert case, the above formula is proved in \cite[Thm. 3.5]{BD}, 
while in the split case the reader may consult \cite[Thm. A, Rmk. 3.3.3]{Sha}, \cite[Pop. 5.5]{GSS} or \cite[Prop.4.25]{LV-IMRN}. 
%

\begin{theorem}\label{LambdaJacobi}
Fix a fundamental index pair $(D_0,r_0)$ and 
let $2k_0$ be a positive even integer. 
Then for all index pair $(D,r)$
such that $p\nmid \Delta=D_0D$ and  
for all even positive integers $2k\in U_{2k_0}$, 
we have 
\[\mathcal{L}_{n,r}(2k)=
\lambda(k)\cdot
\mathcal{E}_p^2
\cdot c_{f^\sharp_{2k}}(n,r),
\]
\end{theorem}

\begin{proof} 
This is an immediate combination of \eqref{Shintani} and 
\eqref{p-adic-L-function}.
%
%
\end{proof}

\section{Big Heegner points and Jacobi forms} 

\subsection{Big Heegner points} 

Let $\mathbb T$ denote Hida's big Galois representation, and $\mathbb T^\dagger$ its {critical twist}, 
introduced in \cite[Def. 2.1.3]{Ho2}. Recall that $\mathbb T^\dagger$ is a free $\mathcal R$-module 
of rank $2$ 
equipped with a continuous $G_\Q:=\Gal(\bar\Q/\Q)$ action, 
such that for each arithmetic point $\varphi\in\mathcal{X}^\mathrm{arith}(\mathcal{R})$, 
the $G_\Q$-representation
\[V_{\varphi}^\dagger:=\mathbb T^\dagger\otimes_\mathcal{R}L_\varphi\] is the self-dual 
twist of the Deligne representation attached to the module form 
$F_\varphi$; here $L_\varphi$ is a finite field extension of $\Q_p$, 
and the tensor product is taken with respect to $\varphi:\mathcal{R}\rightarrow L_\varphi$. 
Assume that the semi-simple 
residual $G_\Q$-representation
$\mathbb T^\dagger/\mathfrak m_\mathcal R\mathbb T^\dagger$, 
where $\mathfrak m_\mathcal R$ is the maximal ideal of $\mathcal R$, is {absolutely irreducible} and $p$-distinguished. 
The definition of $\mathbb T^\dagger$ depends on the choice of a {critical character} 
\[\theta:G_\Q\longrightarrow\Lambda^\times,\] and that there are two possible choices (see \cite[Def. 2.1.3 and Rem. 2.1.4]{Ho2}), related by the quadratic character $x\mapsto\left(\frac{p}{x}\right)$ of conductor $p$; more 
precisely, if we write the cyclotomic character $\epsilon_\mathrm{cyc}$ as 
the product $\epsilon_\mathrm{tame}\cdot\epsilon_\mathrm{wild}$ with $\epsilon_\mathrm{tame}$ taking 
values in $(\Z/p\Z)^\times$ and $\epsilon_\mathrm{wild}$ taking values in $1+p\Z_p$, then 
the two possible critical characters are defined to be 
\[\theta=\left(\frac{p}{}\right)^a\cdot
\epsilon_\mathrm{tame}^{\frac{k_0-2}{2}}\cdot\left[\epsilon_\mathrm{wild}^{1/2}\right]\]
for $a=0$ or $a=1$, 
where $\epsilon_\mathrm{wild}^{1/2}$ is the unique square root of $\epsilon_\mathrm{wild}$ 
taking values in $1+p\Z_p$, and $x\mapsto[x]$ is the inclusion of group-like elements 
$1+p\Z_p\hookrightarrow \Lambda$; here recall that $k_0$ is the weight of our fixed form 
$f_0$ through which the Hida family $f_\infty$ passes. 
We fix the choice of $\theta$ 
for $a=0$ as prescribed in \cite[(4.1) and Rem. 4.1]{Cas}. 

For any extension $H$ of $\Q$, denote by $\Sel(H,\mathbb T^\dagger)$
the {strict Greenberg Selmer group}, whose definition by means of the ordinary filtration of $\mathbb T$ can be found in \cite[Def. 2.4.2]{Ho2}. 

Let $K_D$ be an imaginary quadratic field of discriminant $D<0$; denote by $\mathcal O_D$ its ring of algebraic integers and 
$H_D$ its Hilbert class field. 
Suppose that $D$ is a square $\mod 2N$. 
Fix a residue class 
$r\mod 2N$ such that $r^2\equiv D\mod 4N$ and consider the integral
$\calO_D$-ideal $\mathfrak N=\left(N,\frac{r+\sqrt D}{2}\right)$; then 
$\calO_{D}/\mathfrak{N} \cong \Z/N\Z$. 
Also, take a generator $\omega_D$ of $\mathcal O_K/\Z\simeq\Z$ so that the imaginary part of $\omega_D$ (viewed 
as complex number via our fixed embedding $\bar\Q\hookrightarrow\C$) is positive. Depending on these choices, 
B. Howard in \cite[\S 2.2]{Ho2} constructs a point 
\[\mathfrak X_{D,r}^\mathrm{How}\in H^1(H_D,\mathbb T^\dagger)\]
the {Big Heegner point of conductor $1$} 
associated to the 
quadratic imaginary field $K_D$ (see \cite[Def. 2.2.3]{Ho2}); the notation, slightly different from  {loc. cit.}, reflects the 
dependence on $D$ and $r$. Finally, define 
$\mathfrak Z_{D,r}^\mathrm{How} \in H^1(K_D,\mathbb T^\dagger)$ to be 
the image of $\mathfrak X_{D,r}^\mathrm{How}$ via the corestriction map. 
Recall that all primes dividing $N$ are split in $K_D$, 
and then by \cite[Prop. 2.4.5]{Ho2} we have 
$\mathfrak Z_{D,r}^\mathrm{How}\in \Sel(K_D,\mathbb T^\dagger)$. 

Let $\varphi\in\mathcal{X}^\mathrm{arith}(\mathcal{R})$ and let 
$\chi\mapsto L_p(F_\varphi,\chi)$ denote the Mazur-Tate-Teitelbaum $p$-adic $L$-function of $F_\varphi$, 
viewed as a function on characters $\chi:\Z_p^\times\rightarrow \bar{\Q}_p^\times$. The functional equation satisfied by 
this function is 
\[L_p(F_\varphi,\chi)=-w\chi^{-1}(-N)\theta_\varphi(-N)\cdot L_p(F_\varphi,\chi^{-1}[\cdot]_\varphi),\] 
where $\theta_\varphi$ is the composition of 
the chosen critical character $\theta$ with $\varphi$ and $[\cdot]_\varphi$ is the composition of the 
tautological character $[\cdot]: \Z_p^\times\hookrightarrow\Lambda\hookrightarrow\mathcal R$ with
$\varphi$; here $w=\pm 1$, and is independent of $\varphi$ (but it depends on the choice of $\theta$) and satisfies the 
equation $L_p(F_\varphi,\theta_\varphi)=-wL_p(F_\varphi,\theta_\varphi)$ (see \cite[Prop. 2.3.6]{Ho2}).

In this paper we will work under the following assumption on the
sign of the functional equation of the Mazur-Tate-Teitelbaum $p$-adic $L$-function: 

\begin{assumption} \label{ass1}
$w=1$. \end{assumption}
 
We now discuss the analytic condition that we will assume in this paper to obtain our results.  
First, let $\theta:\Z_p^\times\rightarrow \mathcal R^\times$ be the 
character obtained by factoring $\theta$ through the $p$-cyclotomic extension $\Gal(\Q(\mu_{p^\infty})/\Q)$ 
(where $\mu_{p^\infty}$ is the group of $p$-power roots of unity) and identifying $\Gal(\Q(\mu_{p^\infty})/\Q)$
with $\Z_p^\times$ via the cyclotomic character $\chi_\mathrm{cyc}$; 
thus we have the relation $\theta=\theta\circ\chi_\mathrm{cyc}$. 
Decompose $\Z_p^\times\simeq\Delta\times\Gamma$ with $\Gamma=1+p\Z_p$ and $\Delta=(\Z/p\Z)^\times$ is
identified with the group $\mu_{p-1}$ of $p-1$-roots of unity via the Teichm\"uller character, which 
we denote by $\omega$ as usual. 
The character $\theta:\Z_p^\times\rightarrow\mathcal R^\times$ satisfies the condition 
\[\theta_\varphi(\delta\gamma)=\omega^{k_0-1}(\delta)\cdot\psi^{1/2}(\gamma)\cdot\gamma^{k-1}\] for all arithmetic morphisms 
$\varphi$ of weight $2k$ and wild character $\psi$. 

Following the terminology in \cite[Def. 2]{Ho1}, we say that an arithmetic morphism $\varphi\in\mathcal X^\mathrm{arith}(\mathcal R)$ of weight $2$ is generic for $\theta$ if one of the following conditions hold: 
\begin{enumerate}
\item $F_\varphi$ has non-trivial nebentype; 
\item $F_\varphi$ is the $p$-stabilization of a newform in $S_2(\Gamma_0(N))$, and $\theta_\varphi$ is trivial;  
\item $F_\varphi$ is a newform of level $Np$ and $\theta_\varphi=\omega^{(p-1)/2}$. 
\end{enumerate}
For any $\varphi\in\mathcal X^\mathrm{arith}(\mathcal R)$, define the character $\chi_{D,\varphi}:\mathbb A_{K_D}^\times\rightarrow \bar{\Q}_p^\times$ by the formula 
\[\chi_{D,\varphi}(x):=\theta_\varphi(\mathrm{art}_\Q(\mathrm{N}_{K_D/\Q}(x)));\]
here $\mathrm{art}_\Q$ is the arithmetically normalized Artin map of class field theory, and 
$\mathrm{N}_{K_D/\Q}$ is the norm map.  
For an arithmetic morphism $\varphi\in\mathcal X^\mathrm{arith}(\mathcal R)$ 
which is generic for $\theta$, we say that 
$(F_\varphi,\chi_{D,\varphi})$ has analytic rank $1$ if 
\[\ord_{s=1}L(F_\varphi,\chi_{D,\varphi},s)=1.\] We will work under the following analytic condition: 

\begin{assumption}\label{ass2}
There exists a fundamental discriminant $D_1$ and a weight two prime 
$\varphi_1$ which is generic for $\theta$ 
such that $(F_\varphi,\chi_{D_1,\varphi_1})$ has analytic rank equal to one. 
\end{assumption}

\begin{remark} It is a conjecture of Greenberg that if $\varphi$ is generic for $\theta$ then $(F_\varphi,\chi_{D_1,\varphi_1})$ has possible analytic rank equal to $0$ or $1$ only. Thus, if Greenberg Conjecture holds, Assumption \ref{ass2} is implied by Assumption \ref{ass1}. See \cite[\S3.4]{Ho2} and \cite[Sec. 4]{Ho1} for details. \end{remark}

We fix such a pair $(D_1,\varphi_1)$. Also fix a residue class $r_1\mod 2N$ such that $D_1\equiv r_1^2\mod 4N$. 
As a consequence of this assumption, we see by \cite[Cor. 5]{Ho1} 
that 
$(F_\varphi,\chi_{D_1,\varphi})$ has analytic rank $1$ for all generic arithmetic morphisms $\varphi$, 
except possibly a finite number of them, that $\mathfrak Z_{D_1,r_1}^\mathrm{How}$ is non-$\mathcal R$-torsion, and 
\[
\mathrm{rank}_\mathcal R\left(\Sel(\Q,\mathbb T^\dagger)\right)=1.\]
Further, if we denote by $\mathfrak Z^\mathrm{How}_{D_1,r_1}(\varphi)$ the image via localization map $\mathbb T^\dagger\rightarrow V_\varphi^\dagger$ 
of $\mathfrak Z_{D_1,r_1}^\mathrm{How}$ in the Nekov{\'a}{\v{r}}'s extended Bloch-Kato Selmer group 
$\widetilde H^1_f(\Q,V_\varphi^\dagger)$ attached to the representation $V_\varphi^\dagger$ (introduced in \cite{Nek}), 
we see that $\mathfrak Z^\mathrm{How}_{D_1,r_1}(\varphi)$ is non-zero for all arithmetic morphisms $\varphi$, except possibly a finite number of them; also, it follows from the discussion in \cite[\S 2.4]{Ho2} and \cite[\S 9]{Nek} that, 
for all except a finite number of arithmetic primes, $\mathfrak Z^\mathrm{How}_{D_1,r_1}(\varphi)$ belongs to the Bloch-Kato Selmer 
group $H^1_f(\Q,V_\varphi^\dagger)$, which, again for all arithmetic primes except a finite number of them, 
is equal to the strict Greenberg Selmer group $\Sel(\Q,V_\varphi^\dagger)$ of $V_\varphi^\dagger$
(see \cite{BK} and \cite{Gr} for details on the definitions of these Selmer groups, or \cite[\S 2.4]{Ho2}). 

\subsection{The GKZ Theorem for Heegner cycles} \label{section4.2}

The aim of this Section  is to review results of Hui Xue \cite{Xue} extending 
the Gross-Kohnen-Zagier Theorem \cite[Thm. C]{GKZ}
to modular forms of higher weight. 
  
Fix in this section a newform $f$ of even weight $2k$ and level $\Gamma_0(N)$ with $p\nmid N$. 
Let $K_D$ be a quadratic imaginary field, 
of discriminant $D$, with ring of algebraic integers $\mathcal O_{K_D}$, 
and assume all primes dividing $Np$ are split in $K$. Finally, assume that $w=1$, where $w$ is the sign of the functional equation of 
$L(f,s)$.  

Fix a residue class $r\mod 2N$ such that $r^2\equiv D\mod 4N$. Let 
$z_{D,r}$ be the solution in the upper half plane $\mathcal H$ of the equation $az^2+bz+c=0$, 
where $Q=[a,b,c]\in \mathcal{Q}_{N,D,r}$.  
Then $x_{D,r}=[z_{D,r}]$ is a Heegner point of conductor $1$ in $X_0(N)$, namely, it represents an isogeny of 
elliptic curves with complex multiplication by $\mathcal O_{K_D}$ and cyclic kernel 
of order $N$. If 
$H_{D}$ denotes as above the Hilbert class field of $K_D$, there are $h_{D}=[H_{D}:K_D]$ such points, 
permuted transitively by $G_{D}=\Gal(H_{D}/K_D)$. 
To $x_{D,r}$ we may attach a codimension $k$ in the Chow group of the $(2k-1)$-dimensional Kuga-Sato variety $W_{2k-2}$, which is rational over $H_{D}$, as described in \cite[Sec. 5]{Nek1}. Briefly, one starts 
by considering the elliptic curve $E_{x_{D,r}}$ equipped with a cyclic subgroup of order $N$, 
corresponding to $x_{D,r}$ via the moduli interpretation, and let $\Gamma$ be the graph of of the 
multiplication by $\sqrt{D}$ on $E\times E$; we then consider 
the cycle \[Z(x_{D,r})=\Gamma-E\times\{0\}-D(\{0\}\times E)\] in $E\times E$ 
and define 
\[W(x_{D,r})=\sum_{g\in \Sigma_{2k-2}}\mathrm{sgn}(g) g\left(Z(x_{D,r})^{k-1}\right)\]
where the action of the symmetric group $\Sigma_{2k-2}$ on $2k-2$ letters 
on $E^{2k-2}$ is via permutation. Thanks to the canonical desingularization, 
this defines a cycle, denoted by $W_{D,r}^\mathrm{Heeg}$,  
of codimension $k$ in the Chow group of the $(2k-1)$-dimensional Kuga-Sato variety $W_{2k-2}$, which is rational over $H_{D}$. Adopting a standard notation for Chow groups, we write 
$X_{D,r}\in \mathrm{CH}^{k}(W_{2k-2})_0(H_{D})$. 
We finally define 
$X_{D,r}^\mathrm{Heeg}=\sum_{\sigma\in G_D}W_{D,r}^\mathrm{Heeg}$
and, again in a standard notation, we write 
$X_{D,r}^\mathrm{Heeg}\in \mathrm{CH}^{k}(W_{2k-2})_0(K_{D})$. 
The Heegner cycle considered in \cite[\S2.4]{Zh} and \cite[\S2]{Xue}, 
which we denote by $S_{D,r}^\mathrm{Heeg}$,
is the multiple of $X_{D,r}^\mathrm{Heeg}$ such that the self-intersection 
of $S_{D,r}^\mathrm{Heeg}$ is $(-1)^{k-1}$. Then 
\[S_{D,r}^\mathrm{Heeg}\in  \mathrm{CH}^{k}(W_{2k-2})_0(K_{D})\otimes_{\Q}\R\]
and, since the self-intersection of $Z(x_{D,r})$ is $-2D$ (see for example \cite[\S (3.1)]{Nek2}) in this vector space we have 
$S_{D,r}^\mathrm{Heeg}=X_{D,r}^\mathrm{Heeg}\otimes{|2D|^{-\frac{k-1}{2}}}$, where 
we make the choice of square root of $|2D|$ in $\R$ to be the positive one. 

Let $x\mapsto\overline{x}$ denote the action of the non-trivial element $\tau_D$ of $\Gal(K_D/\Q)$ 
on the Chow group of $K_D$-rational cycles, and define 
\[\left(S_{D,r}^\mathrm{Heeg}\right)^*:=S_{D,r}^\mathrm{Heeg}+\overline{S_{D,r}^\mathrm{Heeg}}.\] 
Denote by $\mathrm{Heeg}_k(X_0(N))$ the $\Z$-submodule of 
$\mathrm{CH}^{k}(W_{2k-2})_0(K_{D})\otimes_{\Q}\R$
generated by 
the elements $\left(S_{D,r}^\mathrm{Heeg}\right)^*$ as $D$ varies. 
Let $\mathbb{T}_N$ be the standard Hecke algebra (over $\Z$) of level $\Gamma_0(N)$, 
and for any $\mathbb{T}_N\otimes_\Z\Q$-module $M$ let $M_f$ denote its $f$-isotypical 
component.  
Let finally 
\[\left(S_{D,r}^\mathrm{Heeg}\right)^*_f\in \mathrm{Heeg}_k(X_0(N))_f\]
be the $f$-isotypical components of 
$\left(S_{D,r}^\mathrm{Heeg}\right)^*$. The assumption $w=1$ implies then that the image 
$S_{D,r,f}^\mathrm{Heeg}$ of 
$S_{D,r}^\mathrm{Heeg}$ in $\left(\mathrm{CH}^{k}(W_{2k-2})_0(K_D)\otimes_\Z\R\right)_f$
belongs to $\left(\mathrm{CH}^{k}(W_{2k-2})_0(\Q)\otimes_\Z\R\right)_f$, and therefore 
$\left(S_{D,r}^\mathrm{Heeg}\right)^*_f=2\cdot S_{D,r,f}^\mathrm{Heeg}$. 

Let \[\phi_f(\tau,z)=\sum_{r^2\leq 4Nn}c_f(n,r)q^n\zeta^r\]
be the Jacobi form corresponding to $f$ under the Skoruppa-Zagier correspondence
\cite{Skoruppa-Zagier}, 
where as usual $q=e^{2\pi i\tau}$ and $\zeta=e^{2\pi i z}$. 

Let 
\begin{equation}\label{Height pairing}
\langle\,,\,\rangle_\R:\mathrm{Heeg}_k(X_0(N))\otimes_\Z\R\times 
\mathrm{Heeg}_k(X_0(N))\otimes_\Z\R \longrightarrow\R\end{equation} 
be the 
restriction to $\mathrm{Heeg}_k(X_0(N))\otimes_\Z\R$ of the height 
pairing defined via arithmetic intersection theory by Gillet-Soul\'e \cite{Gil-Sou}. 
Choose $s_f^*\in \left(\mathrm{Heeg}_k(X_0(N)(\Q))\otimes_\Z\R\right)_f$ 
such that 
\[\langle s_f^*,s_f^*\rangle = \frac{(2k-2)!N^{k-1}}{2^{2k-1}\pi^k(k-1)!\parallel\phi_f\parallel^2}L'(f,k)\]
where $\parallel\phi_f\parallel$ is the norm of $\phi_f$ in the space of Jacobi forms 
equipped with the Petersson scalar product. For the next theorem we need the following 

\begin{assumption}\label{ass3} 
The height pairing $\langle\,,\,\rangle_\R$ in $\eqref{Height pairing}$ is positive definite, for each 
even integer $k\geq 2$. 
\end{assumption} 

%

\begin{theorem}[Xue] \label{Xue}
Assume the height pairing $\langle\,,\,\rangle_\R$ in $\eqref{Height pairing}$ is positive definite. 
For all fundamental discriminants $D$ which are coprime with $2N$ we have  
\[(D)^{\frac{k-1}{2}}\cdot  \left(S_{D,r}^\mathrm{Heeg}\right)^*_f= {c_f\left(\frac{r^2-D}{4N},r\right)}\cdot 
s_f^*.\] 
\end{theorem}

\begin{remark}\label{remark height pairing} The validity of Assumption \ref{ass3} is a consequence 
of the Bloch-Beilinson conjectures: see \cite{Bloch}, 
\cite{Beilinson},  \cite{GS}, and \cite[Conj. 1.1.1 and 1.3.1]{Zhang-Inv}. 
It could be possible to remove this assumption, as suggested 
in \cite{Xue}, using Borcherds's approach \cite{Bor1}, \cite{Bor2} to the GKZ Theorem via singular Theta liftings.  
Actually, Zemel \cite{Zemel} proved such an analogue of Borcherd's results for higher weight modular forms, and therefore it seems reasonable to establish Theorem \ref{Xue} unconditionally using \cite{Zemel}. It would be very interesting to obtain such a result, which however does not seem to be an immediate consequence 
of the methods developed in \cite{Zemel}. 
Indeed, \cite[Thm. 4.6]{Zemel} proves that the generating series of Heegner cycles is a modular form of half-integral weight
with values in $\mathrm{Heeg}_k(X_0(N))\otimes_\Z\Q$.
Using a suitable version of Eichler-Zagier isomorphism between vector-valued half-integral weight and Jacobi forms, and projecting to the $f$-eigencomponent, this result shows that the generating series of Heegner cycles in \cite[Thm. 4.6]{Zemel} gives rise to a 
$\mathrm{Heeg}_k(X_0(N))_f\otimes_\Z\Q$-valued Jacobi form. 
However, to the best knowledge of the authors, the $1$-dimensionality of $\mathrm{Heeg}_k(X_0(N))_f\otimes_\Z\Q$ does not follow directly from the work of Zemel, and therefore any comparison with Xue's result would probably require some new idea. 
 \end{remark}



\subsection{Specialization of Big Heegner points} 
The aim of this Section  is to review the explicit comparison result between Howard Big Heegner points and 
Heegner cycles that could be found conditionally in Castella's thesis \cite{Cas}, announced in Castella-Hsieh \cite{CH} and proved explicitly in \cite{CasHeeg} under a number of arithmetic conditions, the most prominent being that primes $p$ split in the imaginary quadratic field $K_D$.

Let $\varphi\in\mathcal{X}^\mathrm{arith}(\mathcal{R})$ be an arithmetic point of trivial character and weight $2k$, 
and let $f_{2k}^\sharp$ be the form whose $p$-stabilization is $F_\varphi$. 
For any field extension $L/\Q$, we may consider the \'etale Abel-Jacobi map 
\[\Phi^{\text{\'et}}_{W_{2k-2},L}: \mathrm{CH}^k(W_{2k-2})(L)\longrightarrow H^1(L,V^\dagger_\varphi).\]
Let $X_{D,r,2k}^\mathrm{Heeg}$ be the image of 
$X_{D,r}^\mathrm{Heeg}$ in $\left(\mathrm{CH}^{k}(W_{2k-2})_0(K_D)\right)_{f_{2k}^\sharp}$ 
and define 
\[\mathfrak Z_{D,r,2k}^\mathrm{Heeg}:=\Phi^{\text{\'et}}_{W_{2k-2},\Q}\left(X_{D,r,2k}^\mathrm{Heeg}\right).\]
Denote
$\mathfrak Z^\mathrm{How}_{D,r,2k}=\varphi\left(\mathfrak Z^\mathrm{How}_{D,r}\right)$. 
The following result is due to Castella, cf. \cite[Thm 5.5]{CasHeeg}:

\begin{theorem}[Castella] \label{Castella}
Let $\varphi_0$ be an arithmetic point in 
$\mathcal{X}^\mathrm{arith}(\mathcal{R})$ 
with trivial nebentype and even weight $2k_0\equiv 2\mod{p-1}$.
If $p$ is split in $K_D$, then
for any arithmetic point $\varphi$ of even integer $2k>2$ 
and trivial character, 
with $2k\equiv 2k_0\pmod{2(p-1)}$ we have 
\[\mathfrak Z^\mathrm{How}_{D,r,2k}= 
\frac{\left(1-\frac{p^{k-1}}{\varphi(A_p)}\right)^2}{u_D(4|D|)^{\frac{k-1}{2}}}
\cdot
\mathfrak Z_{D,r,2k}^\mathrm{Heeg}
\]
as elements in $H^1(\Q,V_{\kappa}^\dagger)$, where $u_D = | \calO_{K_D}^{\times}|/2$. 
\end{theorem}

Since $S_{D,r}^\mathrm{Heeg}$ is equal to 
$X_{D,r}^\mathrm{Heeg}\otimes{(2D)}^{-\frac{k-1}{2}}$, 
it follows immediately from Theorem \ref{Castella} that for all 
$\varphi$ as in the above theorem we have 
\begin{equation}\label{castella-coro}
2^\frac{k-1}{2}\cdot 2u_D\cdot\mathfrak{Z}_{D,r,2k}^\mathrm{How}=
{\left(1-\frac{p^{k-1}}{a_p(2k)}\right)^2}\cdot
\Phi^{\text{\'et}}_{W_{2k-2},\Q}\left(S_{D,r}^\mathrm{Heeg}\right)_{f_{2k}^\sharp}\end{equation}
as elements in the $\bar\Q_p$-vector space 
$H^1(\Q,V_{f_{2k}^\sharp}^\dagger)\otimes_{\Q}\bar\Q_p$, 
where
we fix an embedding $\bar\Q\hookrightarrow\bar\Q_p$ and, 
with a slight abuse of notation, we write 
$\Phi^{\text{\'et}}_{W_{2k-2},\Q}$ for the $\bar\Q_p$-linear extension of $\Phi^{\text{\'et}}_{W_{2k-2},\Q}$; 
here we use the fact that the element  $S_{D,r}^\mathrm{Heeg}$
belongs to $\mathrm{CH}^{k}(W_{2k-2})_0(\Q)\otimes_\Z\bar\Q$
and not merely in $\mathrm{CH}^{k}(W_{2k-2})_0(\Q)\otimes_\Z\R$. 

 \begin{remark} The referee of this paper suggested that the case of primes $p$ which are inert in $\Q(\sqrt{D})$ could be addressed using a recent result of Daniel Disegni  \cite{Disegni}. The relation between the specialization of Big Heegner points and Heegner cycles is the content of \cite[Thm. C (2)]{Disegni}, which is proved at the end of \S 6.4. The proof of this nice result is based on \cite[Prop. 6.3.6]{Disegni} and \cite[Prop. 3.1.2]{Disegni}, which can be seem as a sort of Hida's Control Theorem. However, since the Control Theorem in question is a comparison between two $1$-dimensional vector spaces, it does not seem enough to prove the relation 
\begin{equation}\label{inert formula}
\mathfrak Z^\mathrm{How}_{D,r,2k}= 
\frac{\left(1-\frac{p^{2k-2}}{a_p^2(2k)} \right)}{u_D(4|D|)^{\frac{k-1}{2}}}
\cdot
\mathfrak Z_{D,r,2k}^\mathrm{Heeg}
\end{equation}
which we conjecture to be true in the inert setting, in analogy with the split case (on the right hand side, $\mathfrak Z_{D,r,2k}^\mathrm{Heeg}$ denotes classical Heegner cycles, see Section  \ref{section4.2}, and $u_D$ denotes half of the units of $K_D$). It might however be very interesting to see if a straightening of the results and methods of \cite{Disegni} could be used to prove the conjectural formula \eqref{inert formula}.\end{remark}

\subsection{The $\Lambda$-adic GKZ Theorem} \label{sec:GKZ}
Suppose from now on that 
Assumptions {\ref{ass1}}, \ref{ass2} and \ref{ass3} 
are satisfied. 

\begin{proposition}\label{InterpolationTheorem} 
Suppose that Assumptions $\ref{ass1}$, $\ref{ass2}$ and $\ref{ass3}$ are satisfied. 
Let $\varphi_0$ be a fixed arithmetic point 
with trivial character and weight $2k_0\equiv 2\mod p-1$. 
Choose a fundamental index $N$ pair $(D_0,r_0)$ such that $p$ is split in $K_{D_0}$. 
Then for any arithmetic point $\varphi$ of even positive integer $2k\equiv 2k_0\mod{p-1}$ 
and trivial character, 
and any index $N$ pair $(D,r)$  with $(D,Np)=1$ 
and $p$ split in $K_D$,
we have 
\[(2D)^{\frac{k-1}{2}}\cdot 2u_D\cdot\mathfrak{Z}_{D,r,2k}^\mathrm{How}
= {\left(1-\frac{p^{k-1}}{a_p(2k)}\right)^2}\cdot
{c_{f_{2k}^\sharp}\left(\frac{r^2-D}{4N},r\right)}\cdot \Phi^{\text{\'et}}_{W_{2k-2},\Q}\left(s_{f_{2k}^\sharp}^*\right) 
\] where $c_{f_{2k}^\sharp}(n,r)$ for the Fourier-Jacobi coefficients of $\mathcal{S}_{D_0,r_0}(f_{2k}^\sharp)$
\end{proposition}

\begin{proof} As a preliminary observation, note that Assumption \ref{ass1} ensures that the sign of the functional equation of the $L$-function of $f_{2k}^\sharp$ is $-1$ (this sign is the same as the sign of the function equation of the $L$-function of $f_{2k}$; see 
for example \cite[Prop. 4]{Ho1}).  
If the index pair $(D,r)$ is fundamental, the result is an obvious consequence of  
Theorem \ref{Xue} and Equation \eqref{castella-coro}, so 
suppose that $D$ is not fundamental, but still $(Np,D)=1$. 
In this case, we may argue as 
in \cite[p. 558]{GKZ}. We start with the above equation for $D$ fundamental, 
and fix an integer $m$ prime to $Np$.  
Multiply both sides by $a_m(2k)$, the
coefficient $q^m$ in $f_{2k}^\sharp$. 
Since $a_m(2k)$ is the eigenvalue of $\mathrm{T}(m)$ acting on $f_{2k}^\sharp$, the left-hand side is 
$(2D)^{\frac{k-1}{2}}\cdot \mathrm{T}(m)\cdot 2u_D\cdot\mathfrak{Z}_{D,r,k}$, while in 
the right-hand side we substitute the factor 
$c_{f_{2k}^\sharp}(n,r)$ with the coefficient 
$c^*_{f^\sharp_{2k}}(n,r)$ of $q^n\zeta^r$ in $\left(\mathcal{S}_{D_0,r_0}(f_{2k}^{\sharp})\right)|\mathrm{T}_J(m)$ (here as usual $D=r^2-4Nn$). The formula in  \cite[p. 508]{GKZ} (top of the page) shows that 
\[a_m(2k)\cdot \mathfrak{Z}_{D,r,2k}^\mathrm{How}= \mathrm{T}_J(m)\cdot 2u_D\cdot\mathfrak{Z}_{D,r,2k}^\mathrm{How}=
\sum_{d\mid m}d^{k-1}\left(\frac{D}{d}\right)2u_D\cdot\mathfrak{Z}_{\frac{Dm^2}{d^2},\frac{rm}{d},k}^\mathrm{How}\] 
and by the equation for the action of Hecke operators in Section \ref{sec:Jacobi}, 
we also have 
\[a_m(\kappa)\cdot c_{f^\sharp_{2k}}(n,r)=
c^*_{f^\sharp_{2k}}(n,r)=\sum_{d\mid m}d^{k-1}\left(\frac{D}{d}\right)c_{f^\sharp_{2k}}(n,r)\left(\frac{nm^2}{d^2},\frac{rm}{d}\right).\]
It follows by induction from the case of fundamental index pair that the equality 
in the statement remains true if $(D,r)$ is replaced by $(Dm^2,rm)$. Now each integer $R$ 
satisfying the congruence $R^2\equiv Dm^2\mod 4N$, also satisfies the congruence 
$R\equiv rm\mod{2N}$ for some 
integer $r$ with $r^2\equiv D\mod{4N}$, which implies the result. 
\end{proof}

\begin{remark}
Tracing back the relation between Big Heegner points and Heegner cycles, 
one can deduce the main result of \cite{Xue} for non-fundamental discriminants 
as well. However, a more direct proof can be also obtained working directly with 
Heegner cycles and following the same argument using \cite[Prop. 2.4.2]{Zh}. 

\end{remark}
\begin{remark} 
Observe that the product on the right hand side of Proposition \ref{InterpolationTheorem} does not depend on the choice of the fundamental index $N$ pair $(D_0,r_0)$, 
even if the last two factors of this product depend on it. \end{remark}

\section{The $p$-adic GKZ Theorem}\label{secGKZ}
\subsection{The main result} 
We suppose in this section that Assumptions \ref{ass1}, \ref{ass2} and \ref{ass3} are satisfied. 
Fix a fundamental index pair $(D_0,r_0)$ such that $p$ split in $K_{D_0}$ and write as above
\[\mathcal{S}_{D_0,r_0}(f_{2k}^\sharp)=
\sum_{r^2\leq 4Nn}c_{f_{2k}^\sharp}(n,r)q^n\zeta^r\]
for an even positive integer $2k$. Fix an even positive integer $2k_0\equiv2\mod{p-1}$.
Recall the domain 
of convergence $U_{2k_0}$ about $2k_0$ introduced in 
Section \ref{sec:GKZ}, and fix 
a connected neighborhood $U\subseteq U_{2k_0}$ in $\mathcal{X}$ 
such that $\lambda(k)\neq 0$ for all $k$ in $U_{2k_0}$.  
Define for $\kappa\in U_{2k_0}$ and $D=r^2-4nN$,  
\[\mathcal{Z}_{n,r}(\kappa)=(2D)^{\frac{\kappa-2}{4}}\cdot 2u_D\cdot\tilde\varphi_{2k}\left(\mathfrak{Z}^\mathrm{How}_{D,r}\right).\] 
Recall the notation $\Sel_\mathcal{K}(\Q,\T^\dagger)=
\Sel(\Q,\T^\dagger)\otimes_\mathcal{R}\mathcal{K}$ 
and \[\Sel_{\mathcal{M}_{2k_0}}(\Q,\T^\dagger)=\Sel_\mathcal{K}(\Q,\T^\dagger)\otimes_\mathcal{K}\mathcal{M}_{2k_0},\] where 
for the second tensor product we use the map $\mathcal{K}\rightarrow\mathcal{M}_{2k_0}$ whose construction is recalled in 
Section \ref{theta-liftings}.

\begin{theorem}\label{AnalyticThm} There exists $\Phi^{\text{\'et}}$ in 
$\Sel_{\mathcal{M}_{2k_0}}(\Q,\mathbb{T})$ such that, in a sufficiently small neighborhood of $2k_0$, we have
$\mathcal{Z}_{n,r}=\mathcal{L}_{n,r}\cdot \Phi^{\text{\'et}}.$ 
\end{theorem}

\begin{proof} Combining Proposition \ref{InterpolationTheorem} with Theorem \ref{LambdaJacobi} (in the split case) 
we obtain 
\[\lambda(k)\cdot\mathcal{Z}_{n,r}(2k)=
\mathcal{L}_{n,r}(2k)\cdot{\Phi^{\text{\'et}}_{W_{2k-2},\Q}\Big(s_{f_{2k}^\sharp}^*\Big)}.
\] 
The result follows by setting  
$\Phi^{\text{\'et}}(\kappa)=\frac{\mathcal{Z}_{n_0,r_0}(\kappa) }{\mathcal{L}_{n_0,r_0}(\kappa)}$
for any index pair with $\mathcal{L}_{n_0,r_0}(2k_0)\neq 0$, 
and using the density of the set 
of even positive integers in $\mathcal{X}$. 
\end{proof}

\begin{remark}
The paper \cite{CasHeeg} shows that 
Big Heegner points are closely related to 
$p$-adic $L$-functions \`a la Bertolini-Darmon-Prasanna, 
and Theorem \ref{AnalyticThm} shows
still another relation of this nature. However, note that the proof of this result makes an essential use of Castella's result. 
\end{remark}

\subsection{Twisted Gross-Kohnen-Zagier} \label{section5.2}
In this Section  we discuss the general conjecture suggested by the results obtained so far. 

Combining Theorems \ref{LambdaJacobi} 
and \ref{AnalyticThm} proves the formula 
 \begin{equation}\label{GKZ1}
(2D)^\frac{2k-2}{4} \cdot 2u_D\cdot\mathfrak{Z}^\mathrm{How}_{D,r,2k}=
c_{f_{2k}^\sharp}(n,r)\cdot
\frac{(2D_0)^{\frac{2k-2}{4}}\cdot 2u_{D_0}\cdot\mathfrak{Z}^\mathrm{How}_{D_0,r_0,2k}}{c_{f_{2k}^\sharp}(n_0,r_0)}
 \end{equation}
for each even positive integer $2k\equiv 2k_0\equiv2\mod p-1$, and all discriminants $D$ with $(D,N)=1$ and $p$ split in $K_D$, where 
we chose an index pair $(D_0,r_0)$ 
such that $\mathcal{L}_{n_0,r_0}(2k_0)\neq 0$ and $p$ splits in $K_{D_0}$. 
%

In \cite{LN-Jacobi} we construct a $p$-adic family of Jacobi forms 
$\mathbb{S}_{D_0,r_0}(\tilde\varphi)=\sum_{n,r}c_{n,r}(\tilde\varphi)q^n\zeta^r$ defined for $\tilde\varphi$ in the metaplectic covering $\tilde{\mathcal{X}}(\mathcal{R})$ of the weight space $\mathcal{X}(\mathcal{R})$, such that 
the specialisation of $\mathbb{S}_{D_0,r_0}(\tilde\varphi)$ at arithmetic points $\tilde\varphi$ lying over arithmetic points $\varphi\in\mathcal{X}^\mathrm{arith}(\mathcal{R})$ interpolates certain theta lifts of the classical forms $F_\varphi$, see \cite[Thm. 5.5]{LN-Jacobi} for details. In particular, 
\cite[Thm. 5.8]{LN-Jacobi} shows the equation 
\begin{equation}\label{GKZ3/2}
c_{n,r}(\tilde\varphi)=\lambda(k)\cdot \varphi(A_p)\cdot\mathcal{E}_2\cdot c_{f_{2k}^\sharp}(n,r) 
\end{equation}
where 
\begin{equation}\label{E2}
\mathcal{E}_2=\left(1-
\frac{2p^{k-1}}{a_p(\kappa)}
-\frac{p^{2k-2}}{a_p^2(\kappa)}
\right),\end{equation}
for each arithmetic point $\varphi\in\mathcal{X}(\mathcal{R})$ of trivial 
character and weight $2k$, where $\tilde\varphi$ is the lift of $\varphi$ 
to the metaplectic covering $\tilde{\mathcal{X}}(\mathcal{R})$ of $\mathcal{X}(\mathcal{R})$ having trivial character (see \cite[Thm. 5.8]{LN-Jacobi}). 

Combining \eqref{GKZ1} and 
\eqref{GKZ3/2}, we obtain for each arithmetic morphism $\varphi\in\mathcal{X}(\mathcal{R})$ of trivial character and weight $2k\equiv2k_0\equiv2\mod{p-1}$
with $2k>2$ (so that $\mathcal{E}_2\neq0$), 
the formula 
\begin{equation}\label{GKZ3}
(2D)^\frac{2k-2}{4} \cdot 2u_D\cdot\mathfrak{Z}^\mathrm{How}_{D,r}(\varphi)=
c_{n,r}(\tilde\varphi)\cdot(2D_0)^\frac{2k-2}{4} \cdot 2u_{D_0}\cdot\mathfrak{Z}^\mathrm{How}_{D_0,r_0}(\varphi)
\end{equation} where we write $\mathfrak{Z}^\mathrm{How}_{D,r}(\varphi)$ for the specialisation of $\mathfrak{Z}^\mathrm{How}_{D,r}$ at $\varphi$. 
 
Assume now that $k_0$ is congruent to $2$ modulo $4$. Then there are two distinct ways to $p$-adically interpolate the term
$(2D)^{\frac{2k-2}{4}}$ appearing in \eqref{GKZ3}, 
by choosing a square root of the critical character $\theta$ used 
to define Big Heegner points. 
Indeed, we may define the square roots $\theta_{a}^{1/2}:\Z_p^{\times} \rightarrow\Lambda^\times$ for $a \in \left\{ 0,1 \right\}$
of the critical character $\theta$ by 
\[\theta_{a}^{1/2}= \Big(\frac{\ }{p}\Big)^{a}  \cdot \epsilon_\mathrm{tame}^{\frac{2k_0-2}{4}}\cdot\left[\epsilon_\mathrm{wild}^{1/4}\right]\]
where $\epsilon_\mathrm{wild}^{1/4}$ is the unique square root of $\epsilon_\mathrm{wild}^{1/2}$ 
taking values in $1+p\Z_p$. Since $k_0$ is congruent to $2$ modulo $4$, the expression  $\epsilon_\mathrm{tame}^{\frac{k_o-2}{4}}$ is without ambiguity. 
For each $\varphi\in\mathcal{X}(\mathcal{R})$ 
we let $\theta_{a,\varphi}^{1/2}$ be the composition of $\theta_{a}^{1/2}$ with the restriction of $\varphi$ 
to $\Lambda^\times$. On the $\mathcal{O}$-module $\Lambda$ we define 
a new structure of $\Lambda$-algebra $\sigma: \Lambda\rightarrow\Lambda$ given by the map $\sigma(t)=t^2$, and we denote this new $\Lambda$-algebra 
by $\tilde{\Lambda}$. 
For any $\tilde\varphi\in\mathcal{X}(\tilde{\mathcal{R}})$ choose $a(\tilde\varphi)\in\{0,1\}$ such that 
the restriction of $\tilde\varphi$ to $\tilde{\Lambda}$ is equal to $\theta^{1/2}_{a(\tilde\varphi),\varphi}$. Then, if $\varphi$ has trivial character and weight $2k$, and $\tilde\varphi$ is the lift of $\varphi$ with trivial character as above, we have $\theta^{1/2}_{a(\tilde\varphi),\varphi}(2D)=(2D)^\frac{2k-2}{4}$. 
It makes then sense to define 
for any  $\tilde{\varphi} \in \calX(\tilde{\mathcal{R}})$, the element 
\[{\mathcal{Z}}_{D,r}(\tilde\varphi)=\theta_{\alpha(\tilde\varphi),\varphi}^{1/2}(2D)\cdot 2u_D\cdot\mathfrak{Z}^\mathrm{How}_{D,r}(\varphi)\] where $\varphi=\pi(\tilde\varphi)$ 
and $\pi:\tilde{\mathcal{X}}(\mathcal{R})\twoheadrightarrow\mathcal{X}(\mathcal{R})$ is the covering map (for the choice of $\varphi=\varphi_{2k}$ with trivial character and $\tilde\varphi$ the lift of $\varphi$ with trivial character, this coincides with the element $\mathcal{Z}_{n,r}(2k)$ defined above, so the new notation is consistent with the old one). Then \eqref{GKZ3} reads as 
\begin{equation}\label{GKZ4}
{\mathcal{Z}}_{D,r}(\tilde\varphi)=c_{n,r}(\tilde\varphi)\cdot
{\mathcal{Z}}_{D_0,r_0}(\tilde\varphi)\end{equation}
We remark that this formula holds under the restrictive conditions that 
$\varphi$ has trivial character and weight $2k\equiv 2k_0\equiv 2\mod{p-1}$, 
$2k_0\equiv 2\mod 4$, $2k>2$, $p$ splits in $D$ (and $D_0$), 
and $\tilde\varphi$ the lift of $\varphi$ with trivial character. However, \eqref{GKZ4} makes sense for all arithmetic 
primes $\tilde\varphi\in\tilde{\mathcal{X}}(\mathcal{R})$ and even if $p$ is inert in $D$ (at least under the condition that $2k_0\equiv 2\mod 4$), 
and it might be viewed as a fuller analogue of the GKZ Theorem over a larger portion of the whole weight space. It is then natural to state the following 
\begin{conjecture} \label{conjecture2} 
Suppose that $2k_0\equiv 2\mod{4}$. Then for all $(D,r)$ and all $\varphi$ in a Zariski open of $\mathcal{X}(\mathcal{R})$ we have 
\[{\mathcal{Z}}_{D,r}(\tilde\varphi)=c_{n,r}(\tilde\varphi)\cdot
{\mathcal{Z}}_{D_0,r_0}(\tilde\varphi).\] 
\end{conjecture}

\begin{remark}\label{rem5.4} If $\varphi$ is a weight two arithmetic prime (with possibly non-trivial character), Conjecture \ref{conjecture} would imply a twisted Gross-Kohnen-Zagier Theorem. More precisely, the specialisation of ${\mathcal{Z}}_{D,r}(\tilde\varphi)$ 
a weight $2$ prime $\varphi$ is a linear combination of Heegner points (see \cite[Sec. 3]{Ho1}, \cite[Sec. 5.1]{Cas}, \cite[Sec. 3.2]{LV1}), and therefore Conjecture \ref{conjecture} says that the ratio between these specialisations is given by coefficients of Jacobi forms, which we interpret as a {twisted GKZ theorem}; recall that {twisted Gross-Zagier theorems} are already available in the literature: see \cite{Ho3}. Moreover, on the opposite direction, note that  if this type of twisted GKZ theorems would be known in the rather general context, similar to that of \cite{Ho3},
 of modular forms of level $\Gamma_1(M)$, then Conjecture \ref{conjecture} would follow from the Zariski density of weight $2$ primes in the weight space. It would be very much interesting to prove such a twisted Gross-Kohnen-Zagier theorems (in parallel to what was done by \cite{Ho3} for twisted Gross-Zagier theorems), and the authors hope to come back in the future to this problem. 
\end{remark} 

\begin{rmk}
Finally, we would like to indicate how our work might shed some light on the main theorem of \cite[Thm. 1.4]{BDPChow}; we thank F. Castella for pointing this out to us. There, it is shown that the cohomology class $\Phi^\text{\'et}(\Delta_t)$ of a generalized Heegner cycle $\Delta_t$ is equal to $m_{D,t} \cdot \sqrt{D} \cdot \delta(P_D)$, where $P_D$ is a point independent of $t$, and $m_{D,t} \in \Z$ satisfies:
\[ m^2_{D,t} = \frac{2 t! ( 2\pi \sqrt{-D})^t}{\Omega^{2t+1}} L(\psi^{2t+1}, t+1),\]

\noindent where $\Omega(A)$ is some complex period, and $L(\psi^{2t+1}, t+1)$ a Hasse-Weil $L$-series attached to a Hecke character $\psi^{2t+1}$. This suggests, in line with the philosophy of the original GKZ theorem, that the coefficients $m_{D,t}$ might have some relationship with Fourier coefficients of a Jacobi form lifting a Hecke theta series.

\end{rmk}

\bibliographystyle{amsalpha}
\bibliography{references}
\end{document}